\documentclass[10pt]{amsart}
\usepackage[active]{srcltx}
\usepackage{amsthm,amsfonts,amsmath,mathrsfs,amssymb}
\usepackage{dsfont}
\usepackage{mathtools}
\usepackage[T1]{fontenc}
\usepackage[utf8]{inputenc}
\usepackage{enumerate}
\usepackage[left=3.5cm,top=4cm,right=3.5cm,bottom=4cm]{geometry}

\usepackage{hyperref}
  \def\N{{\mathbb N}}

\def\({\left(}       \def\){\right)}

\newtheorem{prop}{Proposition}
\newtheorem{lemma}[prop]{Lemma}
\newtheorem{thm}[prop]{Theorem}

\begin{document}
%%% Title
\title[Semigroups of weighted composition operators]{Semigroups of weighted composition operators in spaces of analytic functions}
%%% Information for the first author
\author[I. Ar\'evalo]{Irina Ar\'evalo}
\address{Departamento de Matem\'aticas, Universidad Aut\'onoma de
Madrid, 28049 Madrid, Spain}
\email{irina.arevalo@uam.es}

\author[M. Oliva]{Marcos Oliva}
\address{Departamento de Matem\'aticas, Universidad Aut\'onoma de
Madrid-ICMAT, 28049 Madrid, Spain}
\email{marcos.delaoliva@uam.es}
\thanks{The first author is supported by MTM2015-65792-P from MINECO and FEDER/EU and partially by the Thematic Research Network MTM2015-69323-REDT, MINECO, Spain. The second author is funded by the European Research Council under the grant agreement 307179-GFTIPFD}
%%% General info
%%%%%%%%%%%%%%%%%%%%%%%%%%%
\date{\today}
\begin{abstract}
We study the strong continuity of weighted composition semigroups of the form  $T_tf=\varphi_t'\left(f\circ\varphi_t\right)$ in several spaces of analytic functions. First we give a general result on separable spaces and use it to prove that these semigroups are always strongly continuous in the Hardy and Bergman spaces. Then we focus on two non-separable family of spaces, the mixed norm and the weighted Banach spaces. We characterize the maximal subspace in which a semigroup of analytic functions induces a strongly continuous semigroup of weighted composition operators depending on its Denjoy-Wolff point, via the study of an integral-type operator. 
\end{abstract}
\maketitle

\section{Introduction}

%\subsection{Weighted composition semigroups}
Let $\{T_t:t\geq0\}$ be a family of bounded operators on a Banach space $X$ of analytic functions on the unit disk. The family is a \textbf{semigroup of bounded operators} if it satisfies the following conditions,

\begin{enumerate}[1.]
\item $T_0$ is the identity in the space of bounded operators on $X,$
\item $T_{t+s}=T_tT_s,$ for all $t,s\geq0.$
\end{enumerate}

Moreover, the semigroup is called \textbf{strongly continuous} if $$\lim_{t\to0^+}\|T_tf-f\|_X=0.$$  

One of the most commonly studied semigroups of bounded operators is the semigroup of composition operators, where, for $t\geq0,$ $T_t=C_{\varphi_t},$ and $\{\varphi_t:t\geq0\}$ is a family of analytic self-maps of the disk. Such a family has to satisfy some conditions in order to generate a semigroup of composition operators, namely \begin{enumerate}[1'.]
\item $\varphi_0$ is the identity in $\mathbb{D},$
\item $\varphi_{t+s}=\varphi_t\circ\varphi_s,$ for all $t,s\geq0,$
\item \label{pointwise}$\varphi_t\to\varphi_0$ as $t\to0$ uniformly on compact sets of $\mathbb{D}.$
\end{enumerate}

The family is then called a (one-parameter) \textbf{semigroup of analytic functions}. 

Given a semigroup of analytic functions, the induced family of weighted composition operators $T_t$ given by $T_tf=\varphi_t'\left(f\circ\varphi_t\right)$ is a \textbf{semigroup of weighted composition operators}, as long as they are bounded on $X$. Note that, from now on, we will write $\varphi_t'(z)=\frac{\partial \varphi_t}{\partial z}(z)$ to distinguish it from $\frac{\partial \varphi_t}{\partial t}(z)$. We want to understand operator theory properties of the semigroup of operators, such as spectrum, ideals or dynamics, in terms of geometric function theory of the semigroup of analytic functions, and the first step is to characterize the strong continuity.

The weighted composition operators of this type are related to the isometries of the Hardy space (see \cite{Forelli}) and the semigroups were studied, in the context of the BMOA space, by Stylogiannis in \cite{Sty}. Other semigroups of weighted composition operators, with more general multiplication symbols, were studied by Siskakis in \cite{Siswco}.

In this paper we are interested in the strong continuity of these weighted composition semigroups in several spaces of analytic functions, including the Hardy and Bergman spaces, and non-separable spaces such as the mixed norm spaces and weighted Banach spaces. We will see that there are fundamental differences in the characterization of the strongly continuous weighted composition semigroups depending on the separability properties of the space.

In \cite{BerPor} the authors prove the following basic properties of semigroups of analytic functions that will be useful for us:
\begin{itemize}
\item If $\{\varphi_t\}$ is a semigroup, then each map $\varphi_t$ is univalent.
\item The \textbf{infinitesimal generator} of $\{\varphi_t\}$ is the function $$G(z):=\lim_{t\to0^+}\frac{\varphi_t(z)-z}{t},\,z\in\mathbb{D}.$$ This convergence holds uniformly on compact subsets of $\mathbb{D},$ so $G\in\mathcal{H}(\mathbb{D}).$ The generator satisfies $$G(\varphi_t(z))=\frac{\partial\varphi_t(z)}{\partial t}=G(z)\frac{\partial\varphi_t(z)}{\partial z}$$ and characterizes the semigroup uniquely.
\item The function $G$ has a unique representation $$G(z)=(\overline{b}z-1)(z-b)P(z),\,z\in\mathbb{D},$$ where $P\in\mathcal{H}(\mathbb{D})$ with $\text{Re}\,P\geq0$ in $\mathbb{D}$ and $b\in\overline{\mathbb{D}}$ is the \textbf{Denjoy-Wolff point} of the semigroup, that is, all self-maps in the semigroup share a common Denjoy-Wolff point $b.$
\item If $\{\varphi_t\}$ is non-trivial, there exists a unique univalent function $h:\mathbb{D}\to\mathbb{C},$ called the \textbf{Koenigs function} of $\{\varphi_t\}$ such that:
\begin{itemize}
\item If $b\in\mathbb{D}$ then $h(b)=0,$ $h'(b)=1,$ $$h(\varphi_t(z))=e^{G'(b)t}h(z)$$ for $t\geq0,$ $z\in\mathbb{D}$ and $$h'(z)G(z)=G'(b)h(z),$$ $z\in\mathbb{D}.$ 
\item If $b\in\mathbb{T}$ then $h(0)=0,$ $$h(\varphi_t(z))=h(z)+t$$ for $t\geq0,$ $z\in\mathbb{D}$ and $$h'(z)G(z)=1,$$ $z\in\mathbb{D}.$
\end{itemize} 
\end{itemize}
See also \cite{SisSem} for a review on semigroups of analytic functions and composition operators.
%We will also study an integral-type operator, and thus we will review some properties of the membership of derivatives to the spaces.

The structure of the paper is as follows. Section 2 is an introduction on the spaces of analytic functions that will appear later on. Section 3 will be devoted to the study of general semigroups of weighted composition operators. We first give a result for separable spaces, and prove that in the Hardy and Bergman spaces every semigroup of weighted composition operators is strongly continuous. To study the non-separable case we define the maximal closed linear subspace of a space $X$ such that the semigroup $\{\varphi_t\}$ generates a semigroup of weighted composition operators on it, $$[\varphi_t',X]=\{f\in X: \lim_{t\to0^+}\|\varphi_t'\left(f\circ\varphi_t\right)-f\|_X=0\}.$$ In Sections 4 and 5 we characterize this subspace for the mixed norm spaces and the weighted Banach spaces. We will use the characterizations from Section 3 and, thus, we will need to study the boundedness and compactness of an integral-type operator on such spaces.

Throughout the paper, we will understand $1/\infty$ as zero, the letters $A, B, C, C', K, m$ will denote positive constants, and we will say that two quantities are comparable, denoted by $\alpha\approx\beta,$ if there exist a positive constant $C$ such that $$C^{-1}\alpha\leq\beta\leq C\alpha.$$

\section{Spaces of analytic functions}

\subsection{Hardy and Bergman spaces}

The Hardy space $H^p$ is the space of analytic functions on the unit disk such that its integral means $$M_p(r,f)=\left(\int_0^{2\pi}|f(re^{i\theta})|^p\,\frac{d\theta}{2\pi}\right)^{\frac{1}{p}}$$ for $0<p<\infty$ and $$M_\infty(r,f)=\max_{0\leq\theta<2\pi}|f(re^{i\theta})|$$ are bounded as $r\to 1.$ For every $p,$ $0<p\leq \infty,$ the polynomials are dense in $H^p,$ and if $1\leq p\leq\infty,$ the space $H^p$ is a Banach space with the norm $$\left\|f\right\|_{H^p}=\lim_{r\to1}\left(\frac{1}{2\pi}\int_0^{2\pi}|f(re^{i\theta})|^p d\theta\right)^{1/p},\text{ when }1\leq p<\infty,$$ and $$\left\|f\right\|_{H^\infty}=\sup_{z\in\mathbb{D}}|f(z)|.$$

The Bergman space  $A^p,$ $0<p<\infty,$ is the space of analytic functions on the unit disk such that $$\int_{\mathbb{D}}|f(z)|^p dA(z)<\infty,$$ where $dA$ is the normalized Lebesgue area measure, that is, the subspace of $L^p(\mathbb{D},dA)$ whose elements are analytic functions. For any $0 < p <\infty,$ the Bergman space $A^p$ is a complete space of analytic functions on the unit disk where polynomials are dense. It becomes a Banach space for $1\leq p <\infty$ with the norm $$\|f\|_{A^p}=\left(\int_{\mathbb{D}}|f(z)|^p dz\right)^p.$$ 

More generally, we also define the weighted Bergman space $A^p_\alpha,$ $0<p<\infty,$ $-1<\alpha<\infty,$ of analytic functions on the unit disk such that $$\int_{\mathbb{D}}|f(z)|^p(1-|z|^2)^\alpha\,dA(z)<\infty.$$

The Bergman spaces are closely related to the Hardy spaces. Rewriting the integral condition as $$\frac{1}{\pi}\int_0^1\int_0^{2\pi}|f(re^{i\theta})|^p r d\theta dr = 2\int_0^1M_p^p(r,f) r dr$$ it is clear that every function in the Hardy space $H^p$ belongs also to the Bergman space $A^p.$ Moreover,  Hardy and Littlewood proved $H^p\subseteq A^{2p}$.

\subsection{Mixed norm spaces}

%Let $\mathcal{H}(\mathbb{D})$ be the space of analytic functions on the disk $\mathbb{D}=\{z\in\mathbb{C}:|z|<1\}.$ For $f\in \mathcal{H}(\mathbb{D})$ and $r\in(0,1)$ let $M_p(r,f)$ be the integral mean $$M_p(r,f)=\left(\int_0^{2\pi}|f(re^{i\theta})|^p\,\frac{d\theta}{2\pi}\right)^{\frac{1}{p}}$$ for $0<p<\infty$ and $$M_\infty(r,f)=\max_{0\leq\theta<2\pi}|f(re^{i\theta})|.$$ 

The mixed norm spaces $H(p, q, \alpha),$ $0<p,q\leq\infty,$ $0<\alpha<\infty,$ are the spaces of analytic functions on $\mathbb{D}$ such that $$\int_0^1(1-r)^{\alpha q-1}M_p^q(r,f)\,dr<\infty,$$ for $q<\infty,$ and $$\sup_{0\leq r<1}(1-r)^\alpha M_p(r,f)<\infty$$ for $q=\infty.$
We also define the "little-oh version", $H_0(p,\infty,\alpha),$ as the subspace of functions in $H(p,\infty,\alpha)$ such that $$\lim_{r\to1^-}(1-r)^\alpha M_p(r,f)=0.$$

These spaces are closely related to the Hardy and Bergman spaces, since we can identify the weighted Bergman space $A^p_\alpha,$ $0<p<\infty,$ $-1<\alpha<\infty,$  with the space $H\left(p,p,\frac{\alpha+1}{p}\right)$ and the Hardy space $H^p$ with the limit case $H(p,\infty,0).$ The mixed norm spaces are also related to other spaces of analytic functions, such as Besov and Lipschitz spaces, via fractional derivatives (see \cite[Chapter 7]{JVA}).

They were explicitly defined in Flett's works \cite{Flett}, \cite{Flett2}. Since then, these spaces have been studied by many authors (see \cite{aj}, \cite{Blasco}, \cite{BKV}, \cite{Gab}, \cite{Sledd}). Recently, the mixed norm spaces are mentioned in the works \cite{av}, \cite{avetisyan}, \cite{Ar}, and the monograph \cite{JVA}.

These spaces are Banach for $p,q\geq1,$ and we will denote its norm by $\|f\|_{p,q,\alpha}.$
We have the following results for these spaces \cite[Proposition 7.1.3.]{JVA}:
\begin{prop}	\label{Sarason} For $0\leq r\leq1$, let $f_r(z)=f(rz),$ $z\in\mathbb{D}.$
\begin{itemize}
\item If $f\in H(p,q,\alpha),$ $0<p\leq\infty,$ $0<q,\alpha<\infty,$ then $\|f_r-f\|_{p,q,\alpha}\to0,$ as $r\to1.$
\item If $f\in H_0(p,\infty,\alpha),$ $0<p\leq\infty,$ $0<\alpha<\infty,$ then $\|f_r-f\|_{p,\infty,\alpha}\to0,$ as $r\to1.$ 
\end{itemize}
Moreover, if $f\in H(p,\infty,\alpha)$ and $\|f_r-f\|_{p,\infty,\alpha}\to0,$ as $r\to1,$ then $f\in H_0(p,\infty,\alpha).$
\end{prop}

Notice that, in the language of semigroups of composition operators on mixed norm spaces, Proposition \ref{Sarason} proves that the semigroup of dilations of the unit disk, defined as $\{\varphi_t\}$ with $\varphi_t(z)=e^{-t}z,$ for all $t\geq 0$ and $z\in\mathbb{D}$, induces a strongly continuous semigroup of composition operators on $H(p,q,\alpha)$ for $q<\infty$ and on $H_0(p,\infty,\alpha),$ but not on $H(p,\infty,\alpha).$

In a similar way we can prove that the semigroup of weighted composition operators induced by the same semigroup of analytic functions is strongly continuous on $H(p,q,\alpha)$ for $q<\infty$ and on $H_0(p,\infty,\alpha),$ but not on $H(p,\infty,\alpha).$ Notice first that the operators $$T_tf(z)=\varphi'_t(z)f(\varphi_t(z))=e^{-t}f(e^{-t}z)$$ are bounded on every mixed norm space. 

\begin{thm}	\label{Sarasonwco} Let $\varphi_t(z)=e^{-t}z,$ $t\geq 0$ and $z\in\mathbb{D}.$
\begin{itemize}
\item If $f\in H(p,q,\alpha),$ $0<p\leq\infty,$ $0<q,\alpha<\infty,$ then $\|T_tf-f\|_{p,q,\alpha}\to0,$ as $t\to0.$
\item If $f\in H_0(p,\infty,\alpha),$ $0<p\leq\infty,$ $0<\alpha<\infty,$ then $\|T_tf-f\|_{p,\infty,\alpha}\to0,$ as $t\to0.$ 
\end{itemize}
Moreover, if $f\in H(p,\infty,\alpha)$ and $\|T_tf-f\|_{p,\infty,\alpha}\to0,$ as $t\to0,$ then $f\in H_0(p,\infty,\alpha).$
\end{thm}

\begin{proof}
Denote by $f_t$ the function $f_t(z)=f(e^{-t}z),$ $z\in\mathbb{D}.$ Since $e^{-t}f_t$ converges uniformly to $f$ as $t\to0$ on the compact subset $|z|\leq r,$ then $$M_p(r,e^{-t}f_t-f)\to0$$ as $t\to0.$
On the other hand, since
\begin{align*}
M_p(r,e^{-t}f_t-f)&\leq 4\left(M_p(r,e^{-t}f_t)+M_p(r,f)\right)\leq4\left(M_p(r,f_t)+M_p(r,f)\right)\\&=4\left(M_p(e^{-t}r,f)+M_p(r,f)\right)\leq 8M_p(r,f),
\end{align*}
the Lebesgue's Dominated Convergence theorem yields 
$$\|T_tf-f\|_{p,q,\alpha}^q=\alpha q\int_0^1(1-r)^{\alpha q-1}M_p^q(r,e^{-t}f_t-f)\,dr\to0$$ as $t\to0.$

If $f\in H_0(p,\infty,\alpha),$ fix $\varepsilon>0$ and let $r_0>0$ be such that $\sup_{r\geq r_0}(1-r)^\alpha M_p(r,f)<\varepsilon/16.$ Let $t_0$ be such that, for every $0<t<t_0$ and $r<r_0,$ $$|e^{-t}f(re^{i\theta-t})-f(re^{i\theta})|<\varepsilon/2.$$ Then 
\begin{align*}
\|T_tf-f\|_{p,\infty,\alpha}&\leq\sup_{r<r_0}(1-r)^\alpha M_p(r,e^{-t}f_t-f)+\sup_{r\geq r_0}(1-r)^\alpha M_p(r,e^{-t}f_t-f)\\&\leq\frac{\varepsilon}{2}\sup_{r<r_0}(1-r)^\alpha +8\sup_{r\geq r_0}(1-r)^\alpha M_p(r,f)\leq \varepsilon.
\end{align*}

Finally, let $f\in H(p,\infty,\alpha)$ such that $\|T_tf-f\|_{p,\infty,\alpha}\to0,$ as $t\to0.$ Fix $\varepsilon>0.$ Let $t_0>0$ such that, for every $t<t_0,$ $$\sup_{0<r<1}(1-r)^\alpha M_p(r,e^{-t}f_t-f)<\varepsilon.$$ Then
\begin{align*}
\lim_{r\to1}(1-r)^\alpha M_p(r,f)&\leq4\left(\lim_{r\to1}(1-r)^\alpha M_p(r,e^{-t}f_t)+\lim_{r\to1}(1-r)^\alpha M_p(r,e^{-t}f_t-f)\right)
\\&\leq 4\left(\|f_t\|_{H^\infty}\lim_{r\to1}(1-r)^\alpha+\varepsilon\right)=4\varepsilon.
\end{align*}
\end{proof}

Notice that, in the proof that if $f\in H(p,\infty,\alpha)$ and $\|T_tf-f\|_{p,\infty,\alpha}\to0,$ as $t\to0,$ then $f\in H_0(p,\infty,\alpha)$ we have used that $f_t$ is bounded, and therefore the argument above is not valid if the semigroup of analytic functions $\{\varphi_t\}$ has radial limits of modulus one. 

A consequence of Proposition \ref{Sarason} is that polynomials are dense in $H(p,q,\alpha),$ $0<~p\leq\infty,$ $0<q,\alpha<\infty$ and $H_0(p,\infty,\alpha),$ $0<p\leq\infty,$ $0<\alpha<\infty.$ The closure in $H(p,\infty,\alpha)$ of the set of all analytic polynomials is $H_0(p,\infty,\alpha).$

This result was also proved by Lusky in \cite{Lus96} in a more general setting. He also proved the following theorem.

\begin{thm}\label{dualmixed}
The space $H(p,q,\alpha)$ is reflexive for $1<q<\infty$ and $$H_0(p,\infty,\alpha)^{**}=H(p,\infty,\alpha).$$
\end{thm}

We list some properties of the functions in these spaces that will be needed later. They can be found in \cite{Ar}.

\begin{prop}\label{growth}
Let $0<p,q\leq\infty$ and $0<\alpha<\infty.$ There exists $C>0$ such that for every $f\in H(p,q,\alpha)$ and $z\in\mathbb{D},$
$$|f(z)|\leq \frac{C\|f\|_{p,q,\alpha}}{(1-|z|)^{\alpha+\frac{1}{p}}}.$$
Moreover, for every $z\in\mathbb{D}$ the function $$f_z(w)=\frac{(1-|z|^2)^{\alpha+\frac{1}{p}}}{(1-\bar{z}w)^{2(\alpha+\frac{1}{p})}}$$ belongs to $H_0(p,\infty,\alpha)$ (and therefore $f_z\in H(p,q,\alpha)$ for $0<q\leq\infty$), $$\|f_z\|_{p,q,\alpha}\approx 1 \text{\quad and \quad} |f_z(z)|=\frac{1}{(1-|z|^2)^{\alpha+\frac{1}{p}}}.$$
\end{prop}

%In particular, if we denote the point-evaluation functional as $\delta_z,$ $\delta_z(f)=f(z)$, then for every $z\in \mathbb{D}$ $$\|\delta_z\|\approx|f_z(z)|\approx\frac{1}{(1-|z|)^{\alpha+\frac{1}{p}}}.$$

In the subsequent work, the following result will be very useful for us, since it relates the membership of derivatives of functions belonging to a mixed norm space to another mixed norm space. It is based on a theorem by Hardy and Littlewood (see \cite[Thm.~5.5]{Dur}). 

\begin{lemma}\label{derivative}
For $0< p\leq\infty$ and $\alpha>0,$
\begin{enumerate}[{(1)}]
\item\label{HLO} $M_p(r,f)=O(1-r)^{-\alpha}\Leftrightarrow M_p(r,f')=O(1-r)^{-(\alpha+1)}$ (that is, $f\in H(p,\infty,\alpha)$ if and only if $f'\in H(p,\infty,\alpha+1)$).
\item\label{HLo} $M_p(r,f)=o(1-r)^{-\alpha}\Leftrightarrow M_p(r,f')=o(1-r)^{-(\alpha+1)}$ (that is, $f\in H_0(p,\infty,\alpha)$ if and only if $f'\in H_0(p,\infty,\alpha+1)$).
\end{enumerate}
\end{lemma}

Notice that this means that if we denote by $D$ the differentiation operator, $Df(z) = f'(z),$ $z\in\mathbb{D},$ then it is bounded from $H(p,\infty,\alpha)$ to $H(p,\infty,\alpha + 1)$ and from $H_0(p,\infty,\alpha)$ to $H_0(p,\infty,\alpha + 1).$ 

\subsection{Weighted Banach spaces}

A function $v:\mathbb{D}\to \mathbb{R}_+$ is a \textbf{weight} if it is a a bounded, continuous, positive, and radial function. The weighted Banach spaces with weight $v$ are the spaces $$H_v^\infty=\{f\in H(\mathbb{D}):\sup_{z\in\mathbb{D}}v(z)|f(z)|<\infty\}$$ and $$H_v^0=\{f\in H_v^\infty:\lim_{|z|\to1}v(z)|f(z)|=0\}.$$ These spaces appear naturally in the study of the growth of analytic functions, see, for instance, \cite{RuShi}, \cite{ShiWi}, \cite{ShiWi2}, \cite{AnDun}. They are Banach with respect to the norm $$\|f\|_v=\sup_{z\in\mathbb{D}}v(z)|f(z)|.$$

If $\limsup_{|z|\to1}v(z)>0$ we have that $H_v^\infty=H^\infty$ and $H_v^0=\{0\}.$ Therefore, we will only be interested in what it is called a \textbf{typical weight}, that is, a weight with $\lim_{|z|\to1}v(z)=0$. The spaces induced by these typical weights satisfy that $(H^0_v)^{**}=H_v^\infty,$ and that polynomials are dense in $H^0_v.$ 

To each weight there is a weight defined via an associated growth condition, called the \textbf{associated weight} $\tilde{v}$ (see \cite{BBT}), $$\tilde{v}(z)=\frac{1}{\sup\{|f(z)|:f\in H^\infty_v, \|f\|_{v}\leq 1\}},$$ $z\in\mathbb{D}.$ In \cite{BBT} the authors prove that for each $z\in\mathbb{D}$ there exists a $f_z\in H_v^\infty$ such that $|f_z(z)|=\frac{1}{\tilde{v}(z)}$ and $\|f_z\|\leq1.$ Moreover, $H_v^\infty=H_{\tilde{v}}^\infty$ and $H_v^0=H_{\tilde{v}}^0.$

As in the case of the mixed norm spaces, we will be interested in the behavior of the derivatives of functions of $H^\infty_v.$ We will require several definitions in order to study whether the derivative of a function in a weighted Banach space belongs to a space of the same family. 

Firstly, the weighted Bloch space $\mathcal{B}_v$ is the space of analytic functions $f$ on the unit disk $\mathbb{D}$ such that $$\sup_{z\in\mathbb{D}}v(z)|f'(z)|<\infty.$$ The corresponding "little-oh" space is the little Bloch space $\mathcal{B}_v^0$ of functions in $\mathcal{B}_v$ satisfying $$\lim_{|z|\to1}v(z)|f'(z)|=0.$$ The weighted Bloch spaces are clearly related to the weighted Banach spaces, since $f\in\mathcal{B}_v$ if and only if $f'\in H_v^\infty.$

Following \cite{ShiWi2}, a weight $v$ has the property (U) if there exists a positive number $\alpha$ such that the function $r\to v(r)/(1-r)^\alpha$ is almost increasing (that is, there exists a constant $C$ such that, for $r_1<r_2$ we have $v(r_1)/(1-r_1)^\alpha\leq Cv(r_2)/(1-r_2)^\alpha$), and it satisfies the property (L) if there exists a positive number $\beta$ such that the function $r \to v(r)/(1-r)^\beta$ is almost decreasing. A weight is called normal if it satisfies both properties (U) and (L). 

The relation with the derivatives was given by Lusky \cite{Lusky} in the following theorem.

\begin{thm}\label{lusky}
Assume that the weight $v$ has property (U). Then $v$ has property (L) if and only if $H_v^0 = \mathcal{B}^0_{(1-r^2)v(r)}.$ In particular, if $v$ is a normal weight then $H_v^0 = \mathcal{B}^0_{(1-r^2)v(r)},$ and by duality, $H_v^\infty = \mathcal{B}^\infty_{(1-r^2)v(r)}.$
\end{thm}

In other words, $f\in H_v^\infty $ if and only if $f'\in H^\infty_{(1-r^2)v(r)}.$ In \cite{BCHMP}, the authors study the Volterra operators on weighted Banach spaces and then apply the results to the semigroups of composition operators on those spaces. Like in our work, they use the fact that, for some weights, derivatives of functions on a weighted Banach space belongs to another weighted Banach space. Following this reference, we say that a weight $v$ is \textbf{quasi-normal} if $H_v^\infty = \mathcal{B}^\infty_{(1-r^2)v(r)}.$ Discussion about different weights, including characterizations of whether a weight is quasi-normal, can be found in \cite{BCHMP}.

%\subsection{Bloch-type spaces}
\section{Semigroups of weighted composition operators in general spaces of analytic functions}

Our first goal is to find the semigroups of analytic functions that induce strongly continuous semigroups of weighted composition operators on Banach spaces of analytic functions for which polynomials are dense. To do this we have the next theorem:

\begin{thm}
Let $\{\varphi_t\}$ be a semigroup of analytic functions and $\{\textbf{T}_t\}$ be the induced bounded operator semigroup on a Banach space of analytic functions $X$ satisfying: 
\begin{enumerate}
\item polynomials are dense in $X$,
\item for every $f,g\in X$ such that $|f|\leq|g|$ in $\mathbb{D}$ then $\|f\|_X\leq C\|g\|_X$ for some constant $C,$
\item $\sup_{0\leq t\leq1}\|\textbf{T}_t\|<\infty$, and
\item $\|\varphi_t-id\|_X\to0$ and $\|\varphi_t'-1\|_X\to0$
\end{enumerate}
Then $\{\textbf{T}_t\}$ is strongly continuous on $X.$
\end{thm}

\begin{proof}
Suppose  $\sup_{0\leq t\leq1}\|\textbf{T}_t\|<\infty.$ Since polynomials are dense by (1), for every $f\in X$ there exists a polynomial $p$ such that $\|p-f\|\leq \varepsilon$. Therefore, \begin{align*}
\|\textbf{T}_tf-f\|_X&=\|\varphi_t'(f\circ\varphi_t)-f\|_X\\&\leq\|\varphi_t'(f\circ\varphi_t)-\varphi_t'(p\circ\varphi_t)\|_X+\|\varphi_t'(p\circ\varphi_t)-p\|_X+\|p-f\|_X\\&=\|\textbf{T}_t(f-p)\|_X+\|\varphi_t'(p\circ\varphi_t)-p\|_X+\|p-f\|_X\\&\leq (\|\textbf{T}_t\|+1)\|p-f\|_X+\|\varphi_t'(p\circ\varphi_t)-p\|_X.
\end{align*}
Let $p = \sum_{n=0}^Na_{n}z^n,$ then $$\varphi_t'(p_k\circ\varphi_t)-p_k=\sum_{n=0}^Na_{n}(\varphi_t'\varphi_t^n-z^n)$$
and from here it is enough to show that, for every $n\geq 0,$ $$\|\varphi_t'\varphi_t^n-~id^n\|_X\to0$$ as $t\to0.$
%To see this, first we have that, since $\varphi_t\to id$ uniformly on compact subsets of the disk, then $\varphi_t'\to 1$ uniformly. Also, 
Notice first that for fixed $n\in\mathbb{Z}_+$ and $z\in\mathbb{D},$ $$\varphi^n_t(z)-z^n=\left(\sum_{k=0}^{n-1}\varphi_t^k(z)z^{n-1-k}\right)(\varphi_t(z)-z)$$ and hence, $$|\varphi_t^n(z)-z^n|\leq n|\varphi_t(z)-z|.$$
Therefore, we have that 
\begin{align*}
|\varphi_t'(z)\varphi_t^n(z)-z^n|&\leq|\varphi_t'(z)\varphi_t^n(z)-\varphi_t^n(z)|+|\varphi_t^n(z)-z^n|\\&\leq \|\varphi_t^n\|_\infty|\varphi_t'(z)-1|+n|\varphi_t(z)-z|
\end{align*}
for $z\in\mathbb{D},$ and from here, by (2) and (4),
$$\|\varphi_t'\varphi_t^n-id^n\|_X\leq \|\varphi_t^n\|_\infty\|\varphi_t'-1\|_X+n\|\varphi_t-id\|_X\to0.$$
\end{proof}

This theorem allows us to prove that any bounded weighted composition semigroup induces a strongly continuous semigroup on a Banach space where polynomials are dense, for instance the classical Hardy and Bergman spaces and the Mixed Norm spaces.

\begin{prop}
Every semigroup of analytic functions for which the family of weighted composition operators $\{T_t\}$ are uniformly bounded induces a strongly continuous semigroup of weighted composition operators on the Hardy spaces $H^p$ and on the mixed norm spaces $H(p,q,\alpha)$ for $1\leq q<\infty$ and $H_0(p,\infty,\alpha)$ (and therefore for every weighted Bergman space $A_\alpha^p$).
\end{prop}

\begin{proof}
Clearly both $H^p$ and $H(p,q,\alpha)$ satisfy properties $(1)$ and $(2),$ while $(3)$ is given by hypothesis, so we only need to check $(4).$ Since $\varphi_t$ tends to $\varphi_0$ uniformly on compact subsets of the disk we have that, likewise, $\varphi_t'\to1$ uniformly on compact subsets. Therefore, for $0<r<1,$ $M_p(r,\varphi_t-\varphi_0)\to 0$ and $M_p(r,\varphi_t'-1)\to 0,$ and we get $(4)$ applying Lebesgue's Dominated Convergence Theorem.
\end{proof}

Once we have studied the separable case, we are interested in the spaces where polynomials are not dense, such as the mixed norm space $H(p,\infty,\alpha)$ and the weighted Banach spaces $H_v^\infty.$ The semigroups of composition operators on these spaces were studied in \cite{ArCoR}, \cite{BCDMPS}, \cite{BCDMS} and \cite{BCHMP}. Following this references, we first prove the existence of a maximal subspace of $X$ where the semigroup of weighted composition operators is strongly continuous.

\begin{prop}
Let $X$ be a Banach space of analytic functions and $\{\varphi_t\}$ a semigroup of analytic functions such that $\sup_{0\leq t\leq 1}\|T_t\|<\infty.$ Then there exists a closed subspace $Y$ of $X$ such that:
\begin{itemize}
\item The induced semigroup $\{T_t\}$ is strongly continuous on $Y.$
\item $T_t(Y)\subset Y$ for every $t\geq0.$
\item Every other subspace of $X$ with the above properties is contained in $Y.$ 
\end{itemize}
\end{prop}

The proof is analogous to \cite[Prop. 5.1]{Sty}. Following this reference, we will denote \textbf{the maximal subspace} $Y$ of the above Proposition as $[\varphi_t',X],$ that is, $$[\varphi_t',X]=\{f\in X: \lim_{t\to0^+}\|T_tf-f\|_X=0\}.$$ We can relate this subspace with the generator of the semigroup of analytic functions. 

\begin{thm}\label{subspace}
Let $\{\varphi_t\}$ be a semigroup of analytic functions with generator $G$. Let $\{T_t\}$ be the induced operator semigroup on $X$ and suppose that $\sup_{0\leq t\leq1}\|T_t\|=M<\infty.$ Then $$[\varphi_t',X]=\overline{\{f\in X:(Gf)'\in X\}}.$$
\end{thm}

\begin{proof}
The proof of $[\varphi_t',X]\subseteq\overline{\{f\in X:(Gf)'\in X\}}$ is similar to the proofs of \cite[Thm. 1]{BCDMPS} and \cite[Thm. 5.3]{Sty}, so we omit it here.

%Denote by $\Delta$ the infinitesimal generator of $\{T_t\}$ acting on $[\varphi_t',X]$ and $D(\Delta)$ its domain, then we will show that if $f\in D(\Delta)$ then $(Gf)'\in X.$ Since for $f\in D(\Delta)$ implies $\Delta(f)\in X$ we have $$\lim_{t\to0^+}\left\|\frac{1}{t}(T_t(f)-f)-\Delta(f)\right\|_X.$$ The boundedness of the point evaluation functionals in $X$ means that the norm convergence in $X$ implies uniform convergence on compact subsets of $\mathbb{D},$ that implies pointwise convergence. Then, for each $z\in\mathbb{D}$ we have 
%\begin{align*}
%\Delta(f)(z)&=\lim_{t\to0^+}\frac{\varphi_t'(z)f(\varphi_t(z))-f(z)}{t}=\left(\frac{\partial\left(\varphi_t'(z)f(\varphi_t(z))\right)}{\partial t}\Big|_{t=0}\right)\\&=\left(\frac{\partial(\varphi_t'(z))}{\partial t}\Big|_{t=0}\right)f(\varphi_0(z))+\frac{\partial(\varphi_t(z))}{\partial t}\Big|_{t=0}f'(\varphi_0(z))\\&=G'(z)f(z)+G(z)f'(z)=(Gf)'(z).
%\end{align*}
%This shows that $\Delta(f)(z)=(Gf)'(z)$ for $z\in\mathbb{D}$ and $$D(\Delta)\subseteq\{f\in X:(Gf)'\in X\}.$$ Since $D(\Delta)$ is dense in $[\varphi_t',X]$ we get to the desired inclusion by taking closures.

To prove $\overline{\{f\in X:(Gf)'\in X\}}\subseteq [\varphi_t',X],$ let $f\in X$ with $(Gf)'\in X.$ If we call $Gf=h,$ then $h'=(Gf)'\in X.$ Note that, by the properties of $G$ seen in the introduction, $$(h(\varphi_s(z))'=\varphi_s'(z)h'(\varphi_s(z))=\frac{1}{G(z)}\frac{\partial \varphi_s(z)}{\partial s}h'(\varphi_s(z))=\frac{1}{G(z)}\frac{\partial(h(\varphi_s(z))}{\partial s},$$ and from here, for $t\geq 0$ and $z\in\mathbb{D},$ 
\begin{align*}
T_tf(z)-f(z)&=\varphi'_t(z)f(\varphi_t(z))-f(z)=\frac{G(\varphi_t(z))}{G(z)}f(\varphi_t(z))-f(z)\\&=\frac{1}{G(z)}(h(\varphi_t(z))-h(z))=\frac{1}{G(z)}\int_0^t\frac{\partial h(\varphi_s(z))}{\partial s}ds=\int_0^t(h(\varphi_s(z))'ds.
\end{align*}

Therefore, for $t<1,$ 
\begin{align*}
\|T_tf-f\|_X&=\|\varphi_t'(f\circ\varphi_t)-f\|_X\leq\int_0^t\|(h\circ\varphi_s)'\|_X\,ds\leq\sup_{0\leq s\leq1}\|(h\circ\varphi_s)'\|_X\cdot t\\&=\sup_{0\leq s\leq1}\|T_s(h')\|_X\cdot t\leq \sup_{0\leq s\leq1}\|T_s\|\|h'\|_X\cdot t\leq M\|h'\|_X\cdot t,
\end{align*}
and $\|T_tf-f\|_X\to0$ as $t\to 0^+.$ Taking closures, we have $$\overline{\{f\in X:(Gf)'\in X\}}\subseteq [\varphi_t',X].$$
\end{proof}

Another useful rewriting of this subspace is given by the following operator: For an analytic function $g$ let $W_g:X\to X$ be the operator $$W_gf(z)=g(z)\int_0^z f(\zeta)\,d\zeta$$ for $f\in \mathcal{H}(\mathbb{D}).$ Clearly, $W_{g}f$ is an analytic function, and the operator can be rewritten as $$W_{g}=M_{g}V_{id},$$ where $M_g$ is the pointwise multiplier operator and $V_g$ is the generalized Volterra operator studied by Pommerenke on the Hardy-Hilbert space (see \cite{Pom}) and later on by Aleman and Siskakis (\cite{ASHp}, \cite{ASAp}) on Hardy and Bergman spaces. 

\begin{prop}\label{subspaceoperator}
Let $\{\varphi_t\}$ be a semigroup with associated generator $G.$ Let $X$ be a Banach space of analytic functions with the properties:
\begin{enumerate}[(i)]
\item $X$ contains the constant functions,
%\item For each $b\in\mathbb{D},$ $f\in X\Leftrightarrow\frac{f(z)-f(b)}{z-b}\in X,$ and 
\item If $\{T_t\}$ is the induced semigroup on $X$ then $\sup_{t\in[0,1]}\|T_t\|<\infty.$
\end{enumerate}
Then $$[\varphi_t', X]=\overline{X\cap(W_{\frac{1}{G}}(X)\oplus \mathbb{C})}.$$
Moreover, $$[\varphi_t', X]=\overline{X\cap(W_{\gamma}(X)\oplus \mathbb{C})},$$ where $$W_{\gamma}f(z)=\frac{1}{(1-z)^2P(z)}\int_0^z f(\zeta)\,d\zeta$$ if the Denjoy-Wolff point of $\{\varphi_t\}$ is $b=1$ and, if $X$ also satisfies $f\in X\Leftrightarrow\frac{f(z)-f(b)}{z-b}\in X$ then $$W_{\gamma}f(z)=\frac{1}{P(z)}\int_0^z f(\zeta)\,d\zeta$$ if $b=0$.
\end{prop}

\begin{proof}
By Theorem \ref{subspace} we are interested in the closure of the set $\{f\in X:(Gf)'\in X\},$ that is, $f(z)=\frac{1}{G(z)}\int_0^zh(\zeta)d\zeta+C$ for some $h\in X$ and constant $C.$ Thus $$\{f\in X:(Gf)'\in X\}=X\cap(W_{\frac{1}{G}}(X)\oplus \mathbb{C}),$$ and we just only need to take closures.

Now, if the Denjoy-Wolff point of $\{\varphi_t\}$ is $b\in\mathbb{T},$ we can compose the semigroup by a rotation and assume $b=1.$ Then, by the representation of $G$ discussed in the introduction, $G(z)=(1-z)^2P(z).$ If $b\in\mathbb{D},$ without lack of generality we can assume $b=0$ and then $G(z)=~-zP(z),$ where $\text{Re}P(z)>0.$ If the space $X$ also satisfies that for each $b\in\mathbb{D},$ $f\in X\Leftrightarrow\frac{f(z)-f(b)}{z-b}\in X,$ then we will need only study $W_{\frac{1}{P}}$.
\end{proof}

In the next sections we will use Theorem \ref{subspace} and Proposition \ref{subspaceoperator} to characterize the semigroups of  weighted composition operators that are strongly continuous on the mixed norm spaces and the weighted Banach spaces. 

\section{Semigroups of weigthed composition operators on Mixed Norm spaces}

Suppose that the semigroup of analytic functions $\{\varphi_t\}$ induces a family of bounded weighted composition operators $\{T_t\}$ on $H(p,\infty,\alpha).$ By Proposition \ref{subspaceoperator}, we are interested on the boundedness of the operator $W_{g}=M_{g}V_{id}$ on $H(p,\infty,\alpha).$ Since it is the product of an integral operator and a multiplier, first we will need the following lemma on multipliers on $H(p,\infty,\alpha).$  

\begin{lemma}\label{multiplier} Let $g$ be an analytic function in the unit disk and $M_g$ the pointwise multiplier with symbol $g$. The following are equivalent:
\begin{enumerate}
\item[(a)] $M_g: H(p,\infty,\alpha-1)\to H(p,\infty,\alpha);$
\item[(b)] $M_g: H_0(p,\infty,\alpha-1)\to H_0(p,\infty,\alpha);$
\item[(c)] $g\in H(\infty,\infty,1).$
\end{enumerate}
\end{lemma}

\begin{proof}
First, assume $g\in H(\infty,\infty,1).$ Therefore $$|g(z)|\leq\frac{\|g\|_{\infty,\infty,1}}{1-|z|}$$ for every $z\in\mathbb{D}.$ Let $f\in H(p,\infty,\alpha-1)$, then 
\begin{align*}
\|gf\|_{p,\infty,\alpha}&=\sup_{0\leq r<1}(1-r)^{\alpha}M_p(r,gf)\\&\leq \|g\|_{\infty,\infty,1}\sup_{0\leq r<1}\frac{(1-r)^{\alpha}}{1-r}M_p(r,f)= \|g\|_{\infty,\infty,1}\,\|f\|_{p,\infty,\alpha-1}.
\end{align*}
The same inequalities show that if $f\in H_0(p,\infty,\alpha-1)$ then $gf\in H_0(p,\infty,\alpha).$

On the other hand, suppose $gf\in H(p,\infty,\alpha)$ for every $f\in H(p,\infty,\alpha-1).$ This means that the operator $M_{g}$ is bounded from $H(p,\infty,\alpha-1)$ into $H(p,\infty,\alpha)$. Let us denote by $M$ the norm of this operator.  Then, by Proposition \ref{growth}, 
$$
|g(z)f(z)|\leq\frac{C\|gf\|_{p,\infty,\alpha}}{(1-|z|)^{\alpha+\frac{1}{p}}}\leq\frac{CM\|f\|_{p,\infty,\alpha-1}}{(1-|z|)^{\alpha+\frac{1}{p}}}.
$$ 
Choosing for every $z\in\mathbb{D}$ $f_z\in H(p,\infty,\alpha-1)$ as $$f_z(w)=\frac{(1-|z|^2)^{\alpha-1+\frac{1}{p}}}{(1-\overline{z}w)^{2(\alpha-1+\frac{1}{p})}},$$ a function that satisfies $|f_z(z)|=(1-|z|^2)^{-(\alpha-1+\frac{1}{p})}$ and $\|f_z\|_{p,\infty,\alpha-1}\approx1,$ we get, for every $z\in\mathbb{D},$ $$\frac{|g(z)|}{(1-|z|^2)^{\alpha-1+\frac{1}{p}}}=|g(z)f_z(z)|\lesssim\frac{CM}{(1-|z|)^{\alpha+\frac{1}{p}}}.$$ From here it is clear that $g\in H(\infty,\infty,1).$ The same argument shows that if $W_g$ is bounded from $H_0(p,\infty,\alpha-1)$ to $H_0(p,\infty,\alpha)$ then $g\in H(\infty,\infty,1).$
\end{proof}

Once we understand the behavior of the multiplier, we can prove the boundedness of the operator $W_g$ on the "big-Oh" space and from the "little-oh" space to itself or to the bigger space. The key in this proposition is the fact that, thanks to Lemma \ref{derivative}, the Volterra operator $V_g$ maps $H(p,\infty,\alpha)$ into $H(p,\infty,\alpha-1),$ and the previous lemma.

\begin{prop}\label{wgbounded} Let $g$ be an analytic function in the unit disk. The following are equivalent:
\begin{enumerate}
\item[(a)] $W_g: H(p,\infty,\alpha)\to H(p,\infty,\alpha);$
\item[(b)] $W_g: H_0(p,\infty,\alpha)\to H_0(p,\infty,\alpha);$
\item[(c)] $W_g: H_0(p,\infty,\alpha)\to H(p,\infty,\alpha);$
\item[(d)] $g\in H(\infty,\infty,1).$
\end{enumerate}
\end{prop}

\begin{proof}
Recall that $W_{g}=M_{g}V_{id}.$ By Lemma \ref{derivative}, it is clear that $$V_{id}:H(p,\infty,\alpha)\to H(p,\infty,\alpha-1).$$ Moreover, by the previous Lemma, if $g\in H(\infty,\infty,1)$ then $$M_{g}:H(p,\infty,\alpha-1)\to H(p,\infty,\alpha)$$ and $$M_{g}:H_0(p,\infty,\alpha-1)\to H_0(p,\infty,\alpha),$$ and from here $W_g$ is bounded on $H(p,\infty,\alpha)$ and on $H_0(p,\infty,\alpha).$ This proves $(d)\Rightarrow (a)$ and $(d)\Rightarrow (b).$ Since $(a)\Rightarrow (c)$ and $(b)\Rightarrow (c)$ are clear, we only need to prove $(c)\Rightarrow (d).$

Now, suppose $W_{g}$ is a bounded operator from $H_0(p,\infty,\alpha)$ to $H(p,\infty,\alpha).$ Since the differentiation operator $D$ is bounded from $H_0(p,\infty,\alpha - 1)$ to $H_0(p,\infty,\alpha)$ by Lemma \ref{derivative}, then the operator $M_g = W_g\circ D$ is bounded from $H_0(p,\infty,\alpha - 1)$ to $H(p,\infty,\alpha).$ By Lemma \ref{multiplier}, this means $g\in H(\infty,\infty, 1).$
%Let $f_z\in H_0(p,\infty,\alpha-1)$ be a function satisfying $|f_z(z)|=(1-|z|^2)^{-(\alpha-1+\frac{1}{p})}$ and $\|f_z\|_{p,\infty,\alpha-1}\approx1,$ then its derivative belongs to $H_0(p,\infty,\alpha)$ by Lemma \ref{derivative}. It is also easy to see that $\|f_z'\|_{p,\infty,\alpha}\lesssim1.$ Then, for any $w\in\mathbb{D},$ $$W_g f_z'(w) = g(w)V_{id}f_z'(w)=g(w)\int_0^w f_z'(\zeta)\,d\zeta=g(w)f_z(w),$$ and $|W_gf_z'(z)|\lesssim\|W_g\|(1-|z|^2)^{-(\alpha+\frac{1}{p})},$ again by Proposition \ref{growth}.

%Let $f_0\in~H_0(p,\infty,\alpha)$ be such that $$V_{id}f_0(w)=\int_0^w f_0(\zeta)\,d\zeta=f_z(w),$$ where $f_z\in H_0(p,\infty,\alpha-1)$ is a function satisfying $|f_z(z)|=(1-|z|^2)^{-(\alpha-1+\frac{1}{p})}$ and $\|f_z\|_{p,\infty,\alpha-1}\approx1.$ We know that this function exists and that $\|f_0\|_{p,\infty,\alpha}\approx\|f_z\|_{p,\infty,\alpha-1}$ because of Lemma \ref{derivative}.

%Hence, as in the proof of Lemma~\ref{multiplier},$$\frac{|g(z)|}{(1-|z|)^{\alpha-1+\frac{1}{p}}}=|g(z)||f_z(z)|=|W_gf_z'(z)|\lesssim\frac{\|W_g\|}{(1-|z|)^{\alpha+\frac{1}{p}}},$$ and therefore $g\in H(\infty,\infty,1).$
\end{proof}

The most useful result to study the semigroups of weighted composition operators will be the next one, that characterizes the boundedness of the operator from the bigger space $H(p,\infty,\alpha)$ to $H_0(p,\infty,\alpha).$ We also find that it is equivalent to the compactness and weakly compactness on $H(p,\infty,\alpha).$

\begin{thm}\label{wgcompact} Let $g\in H(\infty,\infty,1).$ The following statements are equivalent:
\begin{enumerate}
\item[(a)] $W_g: H(p,\infty,\alpha)\to H(p,\infty,\alpha)$ is compact;
\item[(b)] $W_g: H(p,\infty,\alpha)\to H(p,\infty,\alpha)$ is weakly compact;
%\item[(c)] $W_g: H_0(p,\infty,\alpha)\to H_0(p,\infty,\alpha)$ is compact;
%\item[(d)] $W_g: H_0(p,\infty,\alpha)\to H_0(p,\infty,\alpha)$ is weakly compact;
\item[(c)] $W_g: H(p,\infty,\alpha)\to H_0(p,\infty,\alpha);$
%\item[(f)] $W_g: H(p,\infty,\alpha)\to H(p,\infty,\alpha)$ does not fix a copy of $l_\infty$;
%\item[(f)] $W_g: H_0(p,\infty,\alpha)\to H_0(p,\infty,\alpha)$ does not fix a copy of $c_0$;
\item[(d)] $g\in H_0(\infty,\infty,1).$
\end{enumerate}

\end{thm}

To prove this theorem we will need the following result, that can be found in \cite[p. 482]{DunSch}.

\begin{thm}\label{DunSch}
Let $T:X\to Y$ be a bounded linear operator between two Banach spaces $X$ and $Y.$ Then, $T$ is weakly compact if and only if $T^{**}(X^{**})\subset Y.$
\end{thm}

\begin{proof}[Proof (Theorem \ref{wgcompact})]
We first prove that $(a)\Leftrightarrow(d).$ Let $g\in H_0(\infty,\infty,1)$ and $\{f_n\}$ a sequence in the unit ball of $H(p,\infty,\alpha)$ that converges uniformly on compact subsets of the unit disk. Therefore, for fixed $\varepsilon>0$ there is $R<1$ such that $|g(z)|(1-|z|)<\varepsilon/\|V_{id}\|$ for $|z|\geq R$. Moreover, if $r\leq R,$ note that $(1-r)^{\alpha-1} \leq 1$ if $\alpha\geq 1$ and $(1-r)^{\alpha-1}\leq (1-R)^{\alpha-1}$ if $\alpha < 1.$ Since $f_n\to 0$ uniformly on compact subsets, there is $N_0\in \N$ such that $$|f_n(z)|\leq \frac{\varepsilon}{((1-R)^{\alpha-1}+1)||g||_{\infty,\infty,1}}$$ for $n\geq N_0$ and for all $|z|\leq R$. Then, for $r<R,$

\begin{align*} M_p(r,V_{id}f_n)&\leq \left(\int_0^{2\pi}\left(\int_0^{r}|f_n(\zeta)|d\zeta\right)^p\frac{d\theta}{2\pi}\right)^{1/p}\leq \frac{\varepsilon r}{((1-R)^{\alpha-1}+1)||g||_{\infty,\infty,1}}.
\end{align*}

From here,
\begin{align*}
(1-r)^{\alpha}M_p(r,gV_{id}f_n)&\leq(1-r)^{\alpha}\sup_{\theta\in [0,2\pi]}|g(re^{i\theta})|M_p(r,V_{id}f_n)\\&\leq ||g||_{\infty,\infty,1} (1-r)^{\alpha-1}M_p(r,V_{id}f_n)\leq \frac{\varepsilon r(1-r)^{\alpha-1}}{((1-R)^{\alpha-1}+1)}\leq \varepsilon
\end{align*}
for $r<R.$

On the other hand, if $r>R$, then 
\begin{align*}
(1-r)^{\alpha}M_p(r,gV_{id}f_n)&\leq(1-r)^{\alpha}\sup_{\theta\in [0,2\pi]}|g(re^{i\theta})|M_p(r,V_{id}f_n)\\
&\leq ||V_{id}f_n||_{p,\infty, \alpha-1} (1-r)\sup_{\theta\in [0,2\pi]}|g(re^{i\theta})|\leq \varepsilon,
\end{align*}
since the operator $V_{id}$ is bounded. Thus, $\lim ||W_gf_n||_{p,\infty, \alpha}=0$ and the operator is compact.

Now we assume $W_g$ is compact on $H(p,\infty,\alpha).$ Since the differentiation operator $D$ is bounded from $H(p,\infty, \alpha-1)$ to $H(p,\infty,\alpha)$ by Lemma \ref{derivative}, we have that the multiplier $$M_g=W_g\circ D:H(p,\infty,\alpha-1)\to H(p,\infty,\alpha)$$ is compact. Therefore, if $\{f_n\}$ is a sequence in the unit ball of $H(p,\infty,\alpha-1)$ such that $f_n\to0$ as $n\to\infty$ uniformly on compact subsets of $\mathbb{D}$, then $\|M_gf_n\|_{p,\infty,\alpha}\to 0$ as $n\to\infty.$

Suppose that $g\not\in H_0(\infty,\infty,1),$ then there exist a sequence $\{z_n\}$ and a constant $C>0$ such that $|z_n|\to 1$ as $n\to\infty,$ and $$|g(z_n)|\geq \frac{C}{1-|z_n|}.$$ 

Given the sequence $\{z_n\}$ we will take the functions $\{f_{z_n}\}$, where $f_{z_n}\in H(p,\infty,\alpha-1)$ are the functions in Prop. \ref{growth}, $$f_{z_n}(w)=\frac{(1-|z_n|^2)^{\alpha+\frac{1}{p}}}{(1-\overline{z}_nw)^{2(\alpha+\frac{1}{p})}},$$ $n\in\mathbb{N},$ $w\in\mathbb{D}.$ These functions satisfy $|f_{z_n}(z_n)|=(1-|z_n|^2)^{-(\alpha-1+\frac{1}{p})}$ and $\|f_{z_n}\|_{p,\infty,\alpha-1}\approx1.$ Moreover, $f_{z_n}\to 0$ uniformly on compact subsets of the disk as $n\to\infty$ and thus, since $M_g$ is compact, $\|M_gf_{z_n}\|\to 0$ as $n\to\infty.$ Nevertheless, $$\frac{\|M_gf_{z_n}\|}{(1-|z_n|)^{\alpha+\frac{1}{p}}}\geq|g(z_n)||f_{z_n}(z_n)|\geq\frac{C|f_{z_n}(z_n)|}{1-|z_n|}=\frac{C}{(1-|z_n|)^{\alpha+\frac{1}{p}}},$$ that is, $$\|M_gf_{z_n}\|\geq C>0,$$ getting a contradiction.

Suppose now that $W_g: H(p,\infty,\alpha)\to H_0(p,\infty,\alpha),$ or equivalently, $W_g^{**}: H^{**}(p,\infty,\alpha)\to H^{**}_0(p,\infty,\alpha).$ By Theorem \ref{dualmixed}, $H^{**}_0(p,\infty,\alpha)=H(p,\infty,\alpha),$ and therefore, $W_g^{**}: H^{**}(p,\infty,\alpha)\to H(p,\infty,\alpha).$ By Theorem \ref{DunSch}, this is equivalent to $W_g: H(p,\infty,\alpha)\to H(p,\infty,\alpha)$ being weakly compact. Therefore we have proved $(c)\Leftrightarrow(b).$

%Since $W_g$ is bounded on $H(p,\infty,\alpha)$ then, by Proposition \ref{growth}, $$\frac{\|W_gh_n\|}{(1-|z_n|)^{\alpha+\frac{1}{p}}}\geq |g(z_n)||V_{id}h_n(z_n)|\geq\frac{C|V_{id}h_n(z_n)|}{1-|z_n|}.$$ Now, let $\{h_n\}$ be like in the proof of Proposition \ref{wgbounded}, $V_{id}h_n=f_{z_n}$ for every $n\in\mathbb{N},$ where $f_{z_n}\in H(p,\infty,\alpha-1)$ satisfies $|f_{z_n}(z_n)|=(1-|z_n|^2)^{-(\alpha-1+\frac{1}{p})}$ and $\|f_{z_n}\|_{p,\infty,\alpha-1}\approx1.$ Then $$\frac{\|W_gh_n\|}{(1-|z_n|)^{\alpha+\frac{1}{p}}}\geq\frac{C|V_{id}h_n(z_n)|}{1-|z_n|}\geq\frac{C}{(1-|z_n|)^{\alpha+\frac{1}{p}}},$$ that is, $$\|W_gh_n\|\geq C>0,$$ that is a contradiction.

Now we see $(d)\Rightarrow(c).$ Let $g\in H_0(\infty,\infty,1)$ and $f\in H(p,\infty,\alpha),$ then, since $W_g=M_g\circ V_{id}$ and $V_{id}$ is bounded from $H(p,\infty,\alpha)$ to $H(p,\infty,\alpha-1)$, we have that
 \begin{align*}
\lim_{r\to 1}(1-r)^\alpha M_p(r,gV_{id}f)&\leq \lim_{r\to1}\,(1-r)^\alpha\sup_{\theta\in [0,2\pi]}|g(re^{i\theta})|\,M_p(r,V_{id}f)\\&\leq \|V_{id}f\|_{p,\infty, \alpha-1}\lim_{r\to1}\,(1-r)\sup_{\theta\in [0,2\pi]}|g(re^{i\theta})|=0.
\end{align*}
Therefore, $W_g$ is bounded from $H(p,\infty,\alpha)$ to $H_0(p,\infty,\alpha).$

Finally, to prove $(c)\Rightarrow (d),$ suppose that $W_g$ is bounded from $H(p,\infty,\alpha)$ to $H_0(p,\infty,\alpha).$ Notice that, since the differentiation operator $D$ is bounded from $H(p,\infty,\alpha - 1)$ to $H(p,\infty,\alpha)$ by Lemma \ref{derivative}, then the operator $M_g = W_g\circ D$ is bounded from $H(p,\infty,\alpha - 1)$ to $H_0(p,\infty,\alpha).$ Now, suppose $g\not\in H_0(\infty,\infty,1),$ then there exist a sequence $\{z_n\}$ and a constant $C>0$ such that $|z_n|\to 1$ and $$|g(z_n)|\geq \frac{C}{1-|z_n|}.$$ Let $$f(z)=\frac{1}{(1-z)^{\alpha+\frac{1}{p}-1}},$$ $z\in\mathbb{D},$ then $f\in H(p,\infty,\alpha - 1)\backslash H_0(p,\infty,\alpha -1).$ Then
$$(1-|z_n|)^\alpha M_p(|z_n|,gf)\geq (1-|z_n|)^\alpha\frac{C}{1-|z_n|}M_p(|z_n|,f)\approx \frac{C (1-|z_n|)^{\alpha-1}}{(1-|z_n|)^{\alpha-1}}= C.$$
Therefore, $M_g$ is not bounded from $H(p,\infty,\alpha - 1)$ to $H_0(p,\infty,\alpha)$ if $g\not\in H_0(\infty,\infty,1).$

\end{proof}

Once we have studied the boundedness of the operator $W_g,$ we come back to semigroups. Since $[\varphi_t',H_0(p,\infty,\alpha)]=H_0(p,\infty,\alpha)$ we have that $$H_0(p,\infty,\alpha)\subseteq [\varphi_t',H(p,\infty,\alpha)]\subseteq H(p,\infty,\alpha)$$ and we want to know what properties must the semigroup $\{\varphi_t\}$ satisfy in order to have the equalities. 

First we deal with $$[\varphi_t',H(p,\infty,\alpha)]\subseteq H(p,\infty,\alpha).$$

\begin{thm}
No nontrivial semigroup of analytic functions induces a strongly continuous semigroup of weighted composition operators on $H(p,\infty,\alpha).$ In other words, $$[\varphi_t',H(p,\infty,\alpha)]\subsetneq H(p,\infty,\alpha).$$
\end{thm}

\begin{proof}
Suppose $[\varphi_t',H(p,\infty,\alpha)]=H(p,\infty,\alpha).$ That means $$H(p,\infty,\alpha)\subseteq\{f\in H(p,\infty,\alpha):(Gf)'\in H(p,\infty,\alpha)\}.$$ By Lemma \ref{derivative} that is equivalent to $$H(p,\infty,\alpha)\subseteq\{f\in H(p,\infty,\alpha):Gf\in H(p,\infty,\alpha-1)\},$$ that is, for every $f\in H(p,\infty,\alpha)$ we have $Gf\in H(p,\infty,\alpha-1),$ and by the pointwise boundedness of functions in $H(p,\infty,\alpha-1),$ $$|G(z)f(z)|\leq\frac{\|Gf\|_{p,\infty,\alpha-1}}{(1-|z|)^{\alpha-1+\frac{1}{p}}}\leq\frac{\|M_G\|\|f\|_{p,\infty,\alpha}}{(1-|z|)^{\alpha-1+\frac{1}{p}}}$$ for any $z\in\mathbb{D}.$
Consider now for every $z\in\mathbb{D}$ the function $f_z(w)=\left(\frac{1-|z|^2}{(1-\overline{z}w)^2}\right)^{\alpha+\frac{1}{p}},$ $w\in\mathbb{D}.$ Clearly $f_z\in H(p,\infty,\alpha)$ for any $z\in\mathbb{D},$ $\|f_z\|_{p,\infty,\alpha}\approx 1$ and $|f_z(z)|=(1-|z|^2)^{-\alpha-\frac{1}{p}}.$ Applying the last inequalities to $f_z$ we get $$\frac{|G(z)|}{(1-|z|^2)^{\alpha+\frac{1}{p}}}=|G(z)f_z(z)|\leq\frac{\|M_G\|}{(1-|z|)^{\alpha-1+\frac{1}{p}}}$$ for any $z\in\mathbb{D},$ that means, $$|G(z)|\leq C(1-|z|)$$ for some constant $C$ and for every $z\in\mathbb{D}.$ Letting $|z|\to1$ we have $G\equiv0.$
\end{proof}

Now, for $H_0(p,\infty,\alpha) = [\varphi_t',H(p,\infty,\alpha)],$ assume first that the Denjoy-Wolff point $b\in\mathbb{D}.$ As we have mentioned before, we can assume $b=0$ in this case. 

\begin{thm}\label{denjoy} Let $\{\varphi_t\}$ be a semigroup with Denjoy-Wolff point $b\in\mathbb{D}.$ Then 
$$H_0(p,\infty,\alpha)=[\varphi_t',H(p,\infty,\alpha)]\Leftrightarrow \frac{1}{P}\in H_0(\infty,\infty,1).$$ 
\end{thm}

\begin{proof}
By Proposition \ref{subspaceoperator}, $$[\varphi_t', H(p,\infty,\alpha)]=\overline{H(p,\infty,\alpha)\cap(W_\gamma(H(p,\infty,\alpha))\oplus \mathbb{C})},$$ 
where $W_\gamma$ is the operator $$W_{\gamma}f(z)=\frac{1}{P(z)}\int_0^z f(\zeta)\,d\zeta,$$ $z\in\mathbb{D}.$ The function $P$ has positive real part, and so does $\frac{1}{P}.$ Therefore, $\frac{1}{|P(z)|}\leq\frac{C}{1-|z|}$ and $\frac{1}{P}\in H(\infty,\infty,1).$ Thus, by Proposition \ref{wgbounded} the operator $W_\gamma$ is bounded on $H(p,\infty,\alpha)$ and 
$$W_\gamma(H(p,\infty,\alpha))\oplus \mathbb{C}\subset H(p,\infty,\alpha).$$ 
From here, $$\overline{H(p,\infty,\alpha)\cap(W_\gamma(H(p,\infty,\alpha))\oplus \mathbb{C})}=\overline{W_\gamma(H(p,\infty,\alpha))\oplus \mathbb{C}}$$ and $$H_0(p,\infty,\alpha)=[\varphi_t,H(p,\infty,\alpha)]=\overline{W_\gamma(H(p,\infty,\alpha))\oplus \mathbb{C}}$$ if and only if $W_\gamma(H(p,\infty,\alpha))\subseteq H_0(p,\infty,\alpha).$ By Theorem \ref{wgcompact}, this is equivalent to $\frac{1}{P}\in H_0(\infty,\infty,1).$ 
\end{proof}

Now, for $b=1,$ recall that $$[\varphi_t',H(p,\infty,\alpha)]=\overline{\left\{f\in H(p,\infty,\alpha):((1-z)^2P\,f)'\in H(p,\infty,\alpha)\right\}}$$ or equivalently, by Lemma \ref{derivative}, $$[\varphi_t',H(p,\infty,\alpha)]=\overline{\left\{f\in H(p,\infty,\alpha):(1-z)^2P\,f\in H(p,\infty,\alpha-1)\right\}}.$$

\begin{thm}
For every nontrivial semigroup of analytic functions with Denjoy-Wolff point $b\in\mathbb{T}$ we have $$H_0(p,\infty,\alpha)\subsetneq [\varphi_t',H(p,\infty,\alpha)].$$
\end{thm}

\begin{proof}
Let $P$ be an arbitrary real-part function associated with a generator of a nontrivial semigroup. Since $P\in H(\infty,\infty,1),$ if $f$ is such that $(1-z)^2f\in H(p,\infty,\alpha-2)$ then $f\in [\varphi_t',H(p,\infty,\alpha)].$ Now let $f(z)=\frac{1}{(1-z)^{\alpha+\frac{1}{p}}},$ then $f\in H(p,\infty,\alpha)\backslash H_0(p,\infty,\alpha),$ and $(1-z)^2f\in H(p,\infty,\alpha-2).$ Therefore, $H_0(p,\infty,\alpha)\neq [\varphi_t',H(p,\infty,\alpha)].$
\end{proof}

\section{Semigroups of weigthed composition operators on Weighted Banach spaces}

As stated in Section 2, we are interested in the weighted Banach spaces $H_v^\infty$ with quasi-normal weights, that is, spaces satisfying $$H_v^\infty=\mathcal{B}^\infty_{(1-r^2)v(r)}.$$
Thanks to the weight being quasi-normal, several results in this section are similar to the analogous in the previous section. First, we study the boundedness of the multiplier. 

\begin{lemma}\label{multiplierweighted} Let $v$ be a weight and $g$ an analytic function in the unit disk and $M_g$ the pointwise multiplier with symbol $g$. The following are equivalent:
\begin{enumerate}
\item[(a)] $M_g: H_{v(r)/(1-r)}^\infty\to H_v^\infty;$
\item[(b)] $M_g: H_{v(r)/(1-r)}^0\to H_v^0;$
\item[(c)] $g\in H(\infty,\infty,1).$
\end{enumerate}
\end{lemma}

\begin{proof}
The proof is analogous to the proof of Lemma \ref{multiplier}, using the definition of the norm of $H_v^\infty,$ and the fact that there is a $f_z$ in $H_v^\infty$ such that $|f_z(z)|=1/\tilde{v}(z)$ and $\|f_z\|\leq 1,$ and that $H_v^\infty=H_{\tilde{v}}^\infty.$
\end{proof}

The boundedness of $W_g = M_g\circ V_{id}$ is now clear on $H_v^\infty$ if the weight is quasi-normal, since we can use Theorem \ref{lusky}. 

\begin{prop}\label{wgboundedweighted} Let $v$ be a quasi-normal weight and $g$ an analytic function in the unit disk. The following are equivalent:
\begin{enumerate}
\item[(a)] $W_g: H_v^\infty\to H_v^\infty;$
\item[(b)] $W_g: H_v^0\to H_v^0;$
\item[(c)] $W_g: H_v^0\to H_v^\infty;$
\item[(d)] $g\in H(\infty,\infty,1).$
\end{enumerate}
\end{prop}

\begin{proof}
In this case Theorem \ref{lusky} gives us the boundedness of $$V_{id}:H_v^\infty\to H_{v(r)/(1-r)}^\infty,$$ that, with Lemma \ref{multiplierweighted}, gives $(d)\implies (a).$ It also proves that the function $f_z',$ where $f_z\in H_{v(r)/(1-r)}^0$ is a function satisfying $$|f_z(z)|=\frac{1-|z|}{\tilde{v}(z)}$$ and $\|f_z\|_{v(r)/(1-r)}\approx1,$ belongs to $H_v^0.$ Therefore, if the operator $W_g$ is bounded from $H_v^0$ to $H_v^\infty$ then $$|g(z)|\left|\int_0^z f_0(\zeta)\,d\zeta\right|=|g(z)||f_z(z)|=\frac{|g(z)|(1-|z|)}{\tilde{v}(z)}\leq\frac{CM}{\tilde{v}(z)},$$ and therefore $g\in H(\infty,\infty,1).$

\end{proof}

We will also use the analogue of Theorem \ref{wgcompact}.

\begin{thm}\label{wgcompactweighted} Let $v$ be a quasi-normal weight and $g\in H(\infty,\infty,1).$ The following statements are equivalent:
\begin{enumerate}
\item[(a)] $W_g: H_v^\infty\to H_v^\infty$ is compact;
\item[(b)] $W_g: H_v^\infty\to H_v^\infty$ is weakly compact;
%\item[(c)] $W_g: H_0(p,\infty,\alpha)\to H_0(p,\infty,\alpha)$ is compact;
%\item[(d)] $W_g: H_0(p,\infty,\alpha)\to H_0(p,\infty,\alpha)$ is weakly compact;
\item[(c)] $W_g: H_v^\infty\to H_v^0;$
%\item[(f)] $W_g: H(p,\infty,\alpha)\to H(p,\infty,\alpha)$ does not fix a copy of $l_\infty$;
%\item[(f)] $W_g: H_0(p,\infty,\alpha)\to H_0(p,\infty,\alpha)$ does not fix a copy of $c_0$;
\item[(d)] $g\in H_0(\infty,\infty,1).$
\end{enumerate}

\end{thm}

\begin{proof}
The proof of $(a)\Leftrightarrow(d)\Rightarrow (c)$ follows the same steps as in the mixed norm spaces, using the ideas of the last proof. 

To prove $(c)\Leftrightarrow(b)$ we use that $(H_v^0)^{**}=H_v^\infty$ and Theorem \ref{DunSch}. 

Finally, the proof of $(b)\Rightarrow(a),$ that is, that a weakly compact multiplier is compact on $H_v^\infty$ is Theorem 5.2 of \cite{CHweighted}.
\end{proof}

As in the case of mixed norm spaces, we find that there are not nontrivial strongly continuous semigroups of weighted composition operators on $H_v^\infty.$

\begin{thm}
No nontrivial semigroup of analytic functions induces a strongly continuous semigroup of weighted composition operators on $H_v^\infty.$ In other words, $$[\varphi_t',H_v^\infty]\subsetneq H_v^\infty.$$
\end{thm}

\begin{proof}
By Theorem \ref{lusky}, our problem is to find the semigroups with generator $G$ such that $$H_v^\infty\subseteq\{f\in H_v^\infty :Gf\in  H_{v(r)/(1-r)}^\infty\}.$$ To find the analytic functions $G$ such that $Gf\in  H_{v(r)/(1-r)}^\infty$ for every $f\in H_v^\infty,$ we proceed as in the proof of Lemma \ref{multiplierweighted}. Suppose the multiplier $M_G$ is bounded from $H_{v(z)}^\infty$ to $H_{v(r)/(1-r)}^\infty$, then $$|G(z)f(z)|\leq\frac{\|M_G\|\|f\|_{v}(1-|z|)}{\tilde{v}(z)}.$$ Taking $f_z$ in $H_{v(z)}^\infty$ as a function with $|f_z(z)|=1/\tilde{v}(z)$ and $\|f_z\|=1,$ we get that $$|G(z)f_z(z)|=\frac{|G(z)|}{\tilde{v}(z)}\leq\frac{\|M_G\|(1-|z|)}{\tilde{v}(z)}$$ for any $z\in\mathbb{D},$ and therefore $G\equiv0.$

\end{proof}

Finally we study whether $H_v^0 = [\varphi_t',H_v^0],$ in the case where the Denjoy-Wolff point is $b\in\mathbb{D}$ and in the case where $b\in\mathbb{T}.$ 

\begin{thm}\label{denjoyweighted} Let $\{\varphi_t\}$ be a semigroup with Denjoy-Wolff point $b\in\mathbb{D}.$ Then 
$$H_v^0=[\varphi_t',H_v^\infty]\Leftrightarrow \frac{1}{P}\in H_0(\infty,\infty,1).$$ 
\end{thm}

\begin{proof}
Following the discussion of Theorem \ref{denjoy}, we know that, if $b=0,$ the operator $W_\gamma$ is bounded on $H_v^\infty$ and, by Proposition \ref{wgboundedweighted}, $$W_\gamma(H_v^\infty)\oplus \mathbb{C}\subset H_v^\infty.$$ 
Thus $$\overline{H_v^\infty\cap(W_\gamma(H_v^\infty)\oplus \mathbb{C})}=\overline{W_\gamma(H_v^\infty)\oplus \mathbb{C}}$$ and $$H_v^0=[\varphi_t,H_v^\infty]=\overline{W_\gamma(H_v^\infty)\oplus \mathbb{C}}$$ if and only if $W_\gamma(H_v^\infty)\subseteq H_v^0.$ By Proposition \ref{wgcompactweighted}, this is equivalent to $\frac{1}{P}\in H_0(\infty,\infty,1).$ 
\end{proof}

%Now, for $b=1,$ recall that $$[\varphi_t',H(p,\infty,\alpha)]=\overline{\left\{f\in H(p,\infty,\alpha):((1-z)^2P\,f)'\in H(p,\infty,\alpha)\right\}}$$ or equivalently, by Lemma \ref{derivative}, $$[\varphi_t',H(p,\infty,\alpha)]=\overline{\left\{f\in H(p,\infty,\alpha):(1-z)^2P\,f\in H(p,\infty,\alpha-1)\right\}}.$$

\begin{thm}
For every nontrivial semigroup of analytic functions with Denjoy-Wolff point $b\in\mathbb{T}$ we have $$H_v^0\subsetneq [\varphi_t',H_v^\infty].$$
\end{thm}

\begin{proof}
Suppose $H_v^0 = [\varphi_t',H_v^\infty].$ Since $$[\varphi_t',H_v^\infty]=\overline{\left\{f\in H_v^\infty:(1-z)^2P\,f\in H_{v(r)/(1-r)}^\infty\right\}}$$ and $P\in H(\infty,\infty,1),$ then if $f$ is such that $(1-z)^2f\in H_{v(r)/(1-r)^2}^\infty$ then $f\in [\varphi_t',H_v^\infty].$ If we take $f\in H_v^\infty$ such that $|f_z(z)|=1/\tilde{v}(z),$ we have that $f_z\not\in H_v^0$ and $(1-z)^2f_z\in H_{v(r)/(1-r)^2}^\infty,$ thus $f_z\in [\varphi_t',H_v^\infty].$ Therefore, $H_v^0\neq [\varphi_t',H_v^\infty].$
\end{proof}

\bibliographystyle{amsplain}

\end{document}